\documentclass{amsart}
\usepackage[margin=1in]{geometry}
\usepackage{amsmath, amsfonts, amssymb, amsthm}
\usepackage{graphicx}
\usepackage{hyperref, url}
\usepackage{color}
\usepackage{dsfont}

\newtheorem{thm}{Theorem} 

\newtheorem{preremark}[thm]{Remark}

\newcommand{\R}{\mathbb R}

\title{A dispersive estimate for the multidimensional Burgers equation}
\date{\today}
\author{Luis Silvestre}
\thanks{LS was partially supported by NSF grants DMS-1254332 and DMS-1764285.}

\begin{document}
\begin{abstract}
We study the multi-dimensional Burgers equation $u_t + u u_{x_1} + \dots + u^d u_{x_d} = 0$. We prove that the $L^\infty$ norm of entropy solutions of this equation decays polynomially as $t \to \infty$ in terms of the $L^1$ norm of the initial data only.
\end{abstract}

\maketitle


We consider entropy solutions of the equation 
\begin{equation} \label{e:burgers}
 u_t + u u_{x_1} + \dots + u^d u_{x_d} = 0 \qquad \text{ for } t \in [0,\infty) \times x \in \R^d.
\end{equation}
We prove the following estimate.

\begin{thm}\label{t:main}
Let $u_0 \in L^\infty \cap L^1(\R^d)$ and $u \in L^\infty([0,\infty) \times \R^d) \cap C([0,\infty), L^1(\R^d)$ be the unique entropy solution to the equation \eqref{e:burgers} with initial data $u_0$. Then $u$ satisfies the estimate
\[ \|u(t,\cdot)\|_{L^\infty(\R^d)} \leq C \|u_0\|_{L^1(\R^d)}^{\gamma_0} t^{-d \gamma_0},\]
where $\gamma_0 = (1+ d (d+1)/2)^{-1}$. The constant $C$ depends on dimension only.
\end{thm}

The multi-dimensional Burgers equation \eqref{e:burgers} is the canonical example of a mutidimensional scalar conservation law that satisfies the genuine nonlinearity condition (originally stated in \cite{lions1994kinetic}). As such, it is often used as a test ground for regularity estimates for scalar conservation laws with the strongest possible genuine nonlinearity assumptions. See for example \cite{crippa2008regularizing}, \cite{serre} or \cite{silvestre2017oscillation}.

Very recently, Denis Serre proved in \cite{serre} the following estimate for \eqref{e:burgers}.
\begin{equation} \label{e:serre}
 \|u(t,\cdot)\|_{L^{(d+1)^2/d}} \leq C \|u_0\|_{L^1}^\gamma t^{-\delta}, 
\end{equation}
where
\[ \gamma = \frac{d^2+2d+2}{(1+d)(d^2+d+2)} \ \text{ and } \ \delta = \frac{2d (d^2+d+1)}{(1+d)^2 (d^2+d+2)}.\]
Note that the exponents in \cite{serre} are writen in terms of the space-time dimension $d+1$. Here, we express them in terms of the spacial dimension and that is why the formulas look different. Using this estimate, it is explained in \cite{serre} how one can extend the notion of entropy solutions for unbounded initial data $u_0 \in L^1$.

The estimate \eqref{e:serre} follows from Theorem \ref{t:main} by interpolation since $\|u(t,\cdot)\|_{L^1} \leq \|u_0\|_{L^1}$.

In \cite{silvestre2017oscillation}, we obtained the following result for general scalar conservation laws.
\begin{thm}[Theorem 1.5 in \cite{silvestre2017oscillation}] \label{t:silvestre} Let $u$ be the entropy solution of a genuinely nonlinear scalar conservation law $u_t + a(u) \cdot \nabla u = 0$. Then, there is an exponent $\gamma_0$ depending on the nonlinearity $a$ (see \cite{silvestre2017oscillation} for a specific expression) such that for all $\gamma \in (0,\gamma_0)$, there exists a constant $C = C(\gamma, d, \|u_0\|_{L^\infty})$ such that
\[ \|u(t,\cdot)\|_{L^\infty} \leq C \|u_0\|_{L^1}^\gamma t^{-d \gamma}.\]
\end{thm}

In the case of the multidimensional Burgers equation, a more explicit form of the estimate can be obtained by scaling. This is done in Remark 8.4 in \cite{silvestre2017oscillation}. We obtain
\begin{equation} \label{e:silvestre}
 \|u(t,\cdot)\|_{L^\infty} \leq C(\gamma,d) \|u_0\|_{L^\infty}^{1 - \gamma(1+d(d+1)/2) } \|u_0\|_{L^1}^\gamma t^{-d \gamma},
\end{equation}
for any $\gamma < \gamma_0 := (1+d (d+1)/2)^{-1}$.

The estimate \eqref{e:silvestre} is not good enough to define entropy solutions with unbounded initial data as in \cite{serre} because of its dependence on $\|u_0\|_{L^\infty}$ on the right hand side. This drawback is pointed out in the comments after the main theorem in \cite{serre} together with the fact that the endpoint exponent $\gamma=\gamma_0$ is not reachable in \eqref{e:silvestre} since $C(d,\gamma)$ may not stay bounded as $\gamma \to \gamma_0$. In this short note, we show a quick argument to overcome both drawbacks and obtain Theorem \ref{t:main}. Our proof is based on the application of \eqref{e:silvestre} in dyadic subintervals of the form $(t/2^{j+1}, t/2^j)$ for $j=0,1,2,3,\dots$. Compounding all these inequalities, we obtain Theorem \ref{t:main}.

The advantage of the multidimensional Burgers equation \eqref{e:burgers} over general scalar conservation laws is that there is an explicit global two-parameter scaling that is used in Remark 8.4 in \cite{silvestre2017oscillation} to obtain \eqref{e:silvestre}. It is easy to work out a version of Theorem \ref{t:main} whenever a similar scaling property works (for example in the case of monomial equations as in Section 2.1 in \cite{serre}). It is currently unclear whether there is any form of Theorem \ref{t:main} that applies to more general scalar conservation laws. In particular, we do not know yet if the dependence on $\|u_0\|_{L^\infty}$ can be removed from the right hand side in Theorem \ref{t:silvestre} or if the decay exponent $\gamma = \gamma_0$ is reachable in general.

\noindent {\bf Acknowledgment.} I thank Denis Serre for pointing out this problem to me. This work started as a result of an email exchange between the author and Denis Serre after the paper \cite{serre} was completed.

\begin{proof}[Proof of Theorem \ref{t:main}]
Let $\gamma$ be any exponent in the range $(0,\gamma_0)$ so that \eqref{e:silvestre} applies. Let us define $\theta := 1 - \gamma (1+d(d+1)/2)$. Note that $\theta \in (0,1)$.

Because of the semigroup property, we can apply the estimate \eqref{e:silvestre} to estimate $\|u(t,\cdot)\|_{L^\infty}$ in terms of the $L^\infty$ and $L^1$ norms of $u(t_1,0)$ at any earlier time $t_1$. In particular, for $t_1 = t/2$ we obtain
\begin{align*} 
\|u(t,\cdot)\|_{L^\infty} &\leq C \left( \frac t2 \right)^{-d\gamma}
\|u(t/2,\cdot)\|_{L^1}^\gamma \|u(t/2,\cdot)\|_{L^\infty}^\theta , \\
&\leq C \left( \frac t2 \right)^{-d\gamma}
\|u_0\|_{L^1}^\gamma \|u(t/2,\cdot)\|_{L^\infty}^\theta
\end{align*}
We reapply \eqref{e:silvestre} to estimate $\|u(t/2,\cdot)\|_{L^\infty}$ in terms of $\|u(t/4,\cdot)\|_{L^\infty}$. Then we iteratively apply \eqref{e:silvestre} to estimate $\|u(t/2^j,\cdot)\|_{L^\infty}$ in terms of $\|u(t/2^{j+1},\cdot)\|_{L^\infty}$. In each step, we get
\begin{align*} 
\|u(t/2^j,\cdot)\|_{L^\infty} &\leq C t^{-d\gamma} 2^{(j+1) d \gamma}
\|u_0\|_{L^1}^\gamma \|u(t/2^{j+1},\cdot)\|_{L^\infty}^\theta
\end{align*}
Compounding all these inequalities, after $k+1$ iterations we are left with
\[ \|u(t,\cdot)\|_{L^\infty} \leq C^{\left( \sum_{j=0}^k \theta^j \right)}
\|u_0\|_{L^1}^{\left( \gamma \sum_{j=0}^k \theta^j \right)} t^{-\left(
d \gamma \sum_{j=0}^k \theta^j \right)} 2^{\left( d \gamma
\sum_{j=0}^k (j+1) \theta^j \right)} \|u(t/2^{k+1},\cdot)\|_{L^\infty}^{\left( \theta^{k+1} \right)}. \]
Since $\theta \in (0,1)$ and entropy solutions are bounded, we can pass to the limit $k \to \infty$ and obtain
\begin{align*} 
 \|u(t,\cdot)\|_{L^\infty} &\leq C^{1/(1-\theta)}
\|u_0\|_{L^1}^{\gamma / (1-\theta) } t^{-
d \gamma / (1-\theta) } 2^{ d \gamma / (1-\theta)^2}, \\
&\leq \tilde C \|u_0\|_{L^1(\R^d)}^{\gamma_0} t^{-d \gamma_0}.
\end{align*}
\end{proof}

\bibliographystyle{plain}
\bibliography{dispersive}

\end{document}